\documentclass[reqno]{amsart}
\usepackage{hyperref}
\begin{document}
\title[\hfilneg Quartet fixed point theorems for nonlinear contractions in partially ordered sets \hfil ]
{Quartet  fixed point theorems for nonlinear contractions in partially ordered metric spaces}
\author[ E. Karapinar \hfil  \hfilneg]{Erdal Karapinar}

\address{erdal karap\i nar,  \newline
Department of Mathematics, Atilim University 06836, \.Incek, Ankara, Turkey}
\email{\href{mailto:erdalkarapinar@yahoo.com}{erdalkarapinar@yahoo.com}}
\email{\href{mailto:ekarapinar@atilim.edu.tr}{ekarapinar@atilim.edu.tr}}


\thispagestyle{empty}

\numberwithin{equation}{section}
\newtheorem{theorem}{Theorem}
\newtheorem{acknowledgement}[theorem]{Acknowledgement}
\newtheorem{algorithm}[theorem]{Algorithm}
\newtheorem{axiom}[theorem]{Axiom}
\newtheorem{case}[theorem]{Case}
\newtheorem{claim}[theorem]{Claim}
\newtheorem{conclusion}[theorem]{Conclusion}
\newtheorem{condition}[theorem]{Condition}
\newtheorem{conjecture}[theorem]{Conjecture}
\newtheorem{corollary}[theorem]{Corollary}
\newtheorem{criterion}[theorem]{Criterion}
\newtheorem{definition}[theorem]{Definition}
\newtheorem{example}[theorem]{Example}
\newtheorem{exercise}[theorem]{Exercise}
\newtheorem{lemma}[theorem]{Lemma}
\newtheorem{notation}[theorem]{Notation}
\newtheorem{problem}[theorem]{Problem}
\newtheorem{proposition}[theorem]{Proposition}
\newtheorem{remark}[theorem]{Remark}
\newtheorem{solution}[theorem]{Solution}
\newtheorem{summary}[theorem]{Summary}
\newcommand{\IR}{\mbox{I \hspace{-0.2cm}R}}
\newcommand{\IN}{\mbox{I \hspace{-0.2cm}N}}

\begin{abstract}
The notion of coupled fixed point is introduced in by Bhaskar and Lakshmikantham in \cite{BL_2006}.
Very recently, the concept of tripled fixed point is introduced by Berinde and Borcut \cite{BB2011}.
In this manuscript, by using the mixed $g$ monotone mapping, some new quartet fixed point theorems are obtained.
We also give some examples to support our results.
\end{abstract}

\maketitle

\section{Introduction and Preliminaries}

In 2006, Bhaskar and Lakshmikantham  \cite{BL_2006} introduced the notion of coupled fixed point and proved some fixed point theorem under certain condition. Later, Lakshmikantham and \'Ciri\'c  in \cite{LC_2009} extended these results by defining
of $g$-monotone property. After that many results appeared on coupled fixed point theory  (see e.g. \cite{LT2011,S2010,EK2011CAMWA,EK2011GUJS,CK2010,BCK2011}).

Very recently, Berinde and Borcut \cite{BB2011} introduced the concept of tripled fixed point and
proved some related theorems.
In this manuscript, the quartet fixed point is considered and
by using the mixed $g$-monotone mapping,  existence and uniqueness of quartet fixed point are obtained.

First we recall the basic definitions and results from which  quartet fixed point is inspired.
Let $(X,d)$ be a metric space and $X^2:=X\times X$. Then the mapping $\rho : X^2\times X^2 \rightarrow [0,\infty)$
such that $\rho((x_1,y_1),(x_2,y_2)):=d(x_1,x_2)+d(y_1,y_2)$ forms a metric on $X^2$.
A sequence $(\{x_n\},\{y_n\}) \in X^2$ is said to be a double sequence of $X$.

\begin{definition} (See \cite{BL_2006})
Let $(X, \leq )$ be partially ordered set and $F:X\times X \rightarrow
X$. $F$ is said to have mixed monotone property if $F(x,y)$ is
monotone nondecreasing in $x$ and is monotone non-increasing in $y$,
that is, for any $x,y \in X$,
\[\displaystyle x_1\leq  x_2 \Rightarrow F(x_1,y) \leq  F(x_2,y), \  \ \mbox{for} \ x_1,x_2 \in X, \ \ \mbox{and} \ \]
\[\displaystyle y_1\leq  y_2 \Rightarrow F(x,y_2) \leq  F(x,y_1), \ \mbox{for} \  y_1,y_2 \in X.\]
\label{defn_MMP}
\end{definition}

\begin{definition}(see \cite{BL_2006})
An element $(x,y)\in X \times X$ is said to be a coupled fixed point
of the mapping $F:X\times X \rightarrow X$ if
\[ F(x,y)=x \ \mbox{and} \ \  F(y,x)=y. \]
\label{defn_CFP}
\end{definition}

Throughout this paper,
let $(X, \leq )$ be partially ordered set and $d$ be a  metric on
$X$ such that $(X, d)$ is a complete metric space.
Further, the product spaces $X \times X$ satisfy the
following:
\begin{equation}
 (u,v) \leq  (x,y) \Leftrightarrow  u\leq  x,\  y\leq  v; \ \ \mbox{for all}\ \ (x,y),(u,v) \in X\times X.
  \label{order_couple}
  \end{equation}

The following two results of Bhaskar and Lakshmikantham in \cite{BL_2006} were extended to class of cone metric spaces
in \cite{EK2011GUJS}:

%Theorem start
\begin{theorem}
Let $F:X \times X \rightarrow X$ be a continuous mapping having the mixed monotone property
on $X$. Assume that there exists a $k \in [0,1)$ with
\begin{equation}
d(F(x,y),F(u,v)) \leq \frac{k}{2} \left[d(x,u)+d(y,v)\right], \  \mbox{for all} \ u \leq  x, \ y \leq  v .
\label{eq_thm_s1}
\end{equation}
If there exist $x_0,y_0 \in X$ such that $x_0 \leq  F(x_0,y_0)$ and $F(y_0,x_0)\leq  y_0$,
then, there  exist $x,y \in  X$ such that $x=F(x,y)$ and $y=F(y,x)$.
\label{theorem_21}
\end{theorem}
%Theorem end

%Theorem start
\begin{theorem}
Let $F:X \times X \rightarrow X$ be a  mapping having the mixed monotone property
on $X$. Suppose that $X$ has the following properties:
\begin{enumerate}
\item[$(i)$]  if a non-decreasing sequence $\{x_n\}\rightarrow x$, then $x_n \leq  x, \ \forall n;$
\item[$(i)$]  if a non-increasing sequence $\{y_n\}\rightarrow y$, then $y \leq  y_n, \ \forall n.$
\end{enumerate}
Assume that there exists a $k \in [0,1)$ with
\begin{equation}
d(F(x,y),F(u,v)) \leq \frac{k}{2} \left[d(x,u)+d(y,v)\right], \  \mbox{for all} \ u \leq  x, \ y \leq  v .
\label{eq_thm22_s1}
\end{equation}
If there exist $x_0,y_0 \in X$ such that $x_0 \leq  F(x_0,y_0)$ and $F(y_0,x_0)\leq  y_0$,
then, there  exist $x,y \in  X$ such that $x=F(x,y)$ and $y=F(y,x)$.
\label{theorem_22}
\end{theorem}

Inspired by Definition \ref{defn_MMP}, the following concept of a $g$-mixed monotone mapping introduced by
V. Lakshmikantham and L.\'Ciri\'c \cite{LC_2009}.

\begin{definition}
Let $(X, \leq )$ be partially ordered set and 
$F:X\times X \rightarrow X$ and $g:X\rightarrow X$.
$F$ is said to have mixed $g$-monotone property if $F(x,y)$ is
monotone $g$-non-decreasing in $x$ and is monotone $g$-non-increasing in $y$,
that is, for any $x,y \in X$,
\begin{equation}
g(x_1)\leq  g(x_2) \Rightarrow F(x_1,y) \leq  F(x_2,y), \  \ \mbox{for} \ x_1,x_2 \in X, \ \ \mbox{and} \
\label{eq_LC3}
\end{equation}
\begin{equation}
g(y_1)\leq  g(y_2) \Rightarrow F(x,y_2) \leq  F(x,y_1), \ \mbox{for} \  y_1,y_2 \in X.
\label{eq_LC4}
\end{equation}
\label{defn_gMMP}
\end{definition}
It is clear that Definition 13 reduces to Definition 9 when $g$ is the identity.
\begin{definition}
An element $(x,y)\in X \times X$ is called a couple point of a mapping 
$F:X\times X \rightarrow X$ and $g:X\rightarrow X$ if
\[F(x,y)=g(x), \ \ \ \ F(y,x)=g(y).\]
\label{defn_CCP}
\end{definition}
\begin{definition}
Let $F:X\times X \rightarrow X$ and $g:X\rightarrow X$ where $X \neq \emptyset$. 
The mappings $F$ and $g$ are said to commute
if
\[g(F(x,y))=F(g(x),g(y)), \ \ \ \mbox{for all}\ x,y \in X.\]
\label{defn_com_gF}
\end{definition}

\begin{theorem}
Let $(X, \leq )$ be partially ordered set and $(X,d)$ be a complete metric space
and also $F:X\times X \rightarrow X$ and $g:X\rightarrow X$ where $X \neq \emptyset$.
Suppose that $F$ has the mixed $g$-monotone property and that there exists  a $k\in [0,1)$ with
\begin{equation}
d(F(x,y),F(u,v)) \leq \frac{k}{2}\left[\frac{d(g(x),g(u))+d(g(y),g(v))}{2}\right]
\label{eq_thm_cor_LC5}
\end{equation}
for all $x,y,u,v \in X$ for which $g(x)\leq  g(u)$ and $g(v)\leq  g(y)$.
Suppose $F(X\times X) \subset g(X)$, $g$ is sequentially continuous and commutes with $F$
and also suppose either $F$ is continuous or $X$ has the following property:
\begin{equation}
\mbox{if a non-decreasing sequence} \ \{x_n\} \rightarrow x,\ \mbox{then} \ x_n \leq  x, \ \mbox{for all} \ n,
\label{eq_thm_cor_LC6}
\end{equation}
\begin{equation}
\mbox{if a non-increasing sequence} \ \{y_n\} \rightarrow y,\ \mbox{then} \ y \leq  y_n, \ \mbox{for all} \ n.
\label{eq_thm_cor_LC7}
\end{equation}
If there exist $x_0,y_0 \in X$ such that $g(x_0)\leq  F(x_0,y_0)$ and $g(y_0)\leq  F(y_0,x_0)$, then there exist
$x,y \in X$ such that $g(x)=F(x,y)$ and $g(y)=F(y,x)$, that is, $F$ and $g$ have a couple coincidence.
\label{thm_cor_LC21}
\end{theorem}

%Theorem end

Berinde and Borcut \cite{BB2011} introduced the following partial order on the product space
$X^3=X\times X\times X$:
\begin{equation}
 (u,v,w)\leq (x,y,z) \mbox{ if and only if } x\geq u, \ y\leq v, \ z\geq w,
\end{equation}
where $(u,v,w), (x,y,z) \in X^3$.
Regarding this partial order, we state the definition of the following mapping.
%%%%%%%%%%%%%%%%%%%%%%%%%%%%%%%%%%%%%%%%%%%%%%%%%%%%%%%%%%%%%%%%%%%%%
%%%%%%%%%%%%%%%%%%%%%%%%%%%%%%%%%%%%%%%%%%%%%%%%%%%%%%%%%%%%%%%%%%%%%
\begin{definition} (See \cite{BB2011})
Let $(X, \leq )$ be partially ordered set and $F:X^3 \rightarrow X$.
We say that $F$ has the mixed monotone property if $F(x,y,z)$ is monotone
non-decreasing in $x$ and $z$, and it is monotone non-increasing in $y$, that is, for any $x,y,z \in X$
\begin{equation}
\begin{array}{r}
x_1,x_2 \in X, \ x_1 \leq x_2 \Rightarrow \ F(x_1,y,z) \leq F(x_2,y,z),\\
y_1,y_2 \in X, \ y_1 \leq y_2 \Rightarrow \ F(x,y_1,z) \geq F(x,y_2,z),\\
z_1,z_2 \in X, \ z_1 \leq z_2 \Rightarrow \ F(x,y,z_1) \leq F(x,y,z_2). \\
\end{array}
\label{ttommp}
\end{equation}
\label{ttmmp}
\end{definition}
%%%%%%%%%%%%%%%%%%%%%%%%%%%%%%%%%%%%%%%%%%%%%%%%%%%%%%%%%%%%%%%%%
%%%%%%%%%%%%%%%%%%%%%%%%%%%%%%%%%%%%%%%%%%%%%%%%%%%%%%%%%%%%%%%%%%%%%%%%%
%%%%%%%%%%%%%%%%%%%%%%%%%%%%%%%%%%
\begin{theorem} (See \cite{BB2011})
Let $(X, \leq )$ be partially ordered set and $(X,d)$ be a complete metric space.
Let $F:X \times X \times X  \rightarrow X$ be a  mapping having the mixed monotone property on $X$.
Assume that there exist constants $a,b,c \in [0,1)$ such that $a+b+c<1$
for which
\begin{equation}
d(F(x,y,z),F(u,v,w)) \leq ad(x,u)+bd(y,v)+cd(z,w)
\label{thm2eqmainberinde}
\end{equation}
for all $x\geq u, \ y\leq v, \ z \geq w$. Assume that $X$ has the following
properties:
\begin{itemize}
\item[$(i)$] if non-decreasing sequence $x_n\rightarrow x$, then $x_n\leq x$  for all $n,$
\item[$(ii)$]if non-increasing sequence $y_n\rightarrow y$, then $y_n\geq y$  for all $n$,
\end{itemize}
If there exist $x_0,y_0,z_0 \in X$ such that
\[x_0\leq F(x_0,y_0,z_0), \ \ y_0 \geq F(y_0,x_0,y_0),\ \ \ z_0\leq F(x_0,y_0,z_0)\]
then there exist $x,y,z \in X$ such that
\[F(x,y,z)=x \mbox{  and  } F(y,x,y)=y \mbox{  and  } F(z,y,x)=z \]
\label{mainberinde2}
\end{theorem}

%%%%%%%%%%%%%%%%%%%%%%%%%%%%%%%%%%%%
The aim of this paper is introduce the concept of quartet fixed point and prove the related fixed point theorems.

\section{Quartet Fixed Point Theorems}

Let $(X, \leq )$ be partially ordered set and $(X,d)$ be a complete metric space. We state the definition of the following mapping. Throughout the manuscript we denote $X \times X \times X \times X$ by $X^4$.
%%%%%%%%%%%%%%%%%%%%%%%%%%%%%%%%%%%%%%%%%%%%%%%%%%%%%%%%%%%%%%%%%%%%%
%%%%%%%%%%%%%%%%%%%%%%%%%%%%%%%%%%%%%%%%%%%%%%%%%%%%%%%%%%%%%%%%%%%%%
\begin{definition} (See \cite{KL})
Let $(X, \leq )$ be partially ordered set and $F:X^4 \rightarrow X$.
We say that $F$ has the mixed monotone property if $F(x,y,z,w)$ is monotone
non-decreasing in $x$ and $z$, and it is monotone non-increasing in $y$ and $w$, that is, for any $x,y,z,w \in X$
\begin{equation}
\begin{array}{r}
x_1,x_2 \in X, \ x_1 \leq x_2 \Rightarrow \ F(x_1,y,z,w) \leq F(x_2,y,z,w),\\
y_1,y_2 \in X, \ y_1 \leq y_2 \Rightarrow \ F(x,y_1,z,w) \geq F(x,y_2,z,w),\\
z_1,z_2 \in X, \ z_1 \leq z_2 \Rightarrow \ F(x,y,z_1,w) \leq F(x,y,z_2,w), \\
w_1,w_2 \in X, \ w_1 \leq w_2 \Rightarrow \ F(x,y,z,w_1) \geq F(x,y,z,w_2). \\
\end{array}
\label{ommpa}
\end{equation}
\label{mmpa}
\end{definition}
%%%%%%%%%%%%%%%%%%%%%%%%%%%%%%%%%%%%%%%%%%%%%%%%%%%%%%%%%%%%%%%%%
%%%%%%%%%%%%%%%%%%%%%%%%%%%%%%%%%%%%%%%%%%%%%%%%%%%%%%%%%%%%%%%%%%%%%%%%%
\begin{definition}(See \cite{KL})
An element $(x,y,z,w) \in X^4$ is called a quartet fixed point of $F:X \times X \times X \times X \rightarrow X$ if
\begin{equation}
\begin{array}{rl}
F(x,y,z,w)=x,  & F(x,w,z,y)=y, \\ 
F(z,y,x,w)=z,  & F(z,w,x,y)=w.
\end{array}
\label{tfpa}
\end{equation}
\label{defntfpa}
\end{definition}
\begin{definition}
Let $(X, \leq )$ be partially ordered set and $F:X^4 \rightarrow X$.
We say that $F$ has the mixed $g$-monotone property if $F(x,y,z,w)$ is monotone
$g$-non-decreasing in $x$ and $z$, and it is monotone $g$-non-increasing in $y$ and $w$, 
that is, for any $x,y,z,w \in X$
\begin{equation}
\begin{array}{r}
x_1,x_2 \in X, \ g(x_1) \leq g(x_2) \Rightarrow \ F(x_1,y,z,w) \leq F(x_2,y,z,w),\\
y_1,y_2 \in X, \ g(y_1) \leq g(y_2) \Rightarrow \ F(x,y_1,z,w) \geq F(x,y_2,z,w),\\
z_1,z_2 \in X, \ g(z_1) \leq g(z_2) \Rightarrow \ F(x,y,z_1,w) \leq F(x,y,z_2,w), \\
w_1,w_2 \in X, \ g(w_1) \leq g(w_2) \Rightarrow \ F(x,y,z,w_1) \geq F(x,y,z,w_2). \\
\end{array}
\label{ommp}
\end{equation}
\label{mmp}
\end{definition}
%%%%%%%%%%%%%%%%%%%%%%%%%%%%%%%%%%%%%%%%%%%%%%%%%%%%%%%%%%%%%%%%%
%%%%%%%%%%%%%%%%%%%%%%%%%%%%%%%%%%%%%%%%%%%%%%%%%%%%%%%%%%%%%%%%%%%%%%%%%
\begin{definition}
An element $(x,y,z,w) \in X^4$ is called a quartet  coincidence point of $F:X^4 \rightarrow X$ 
and $g:X\rightarrow X $ if
\begin{equation}
\begin{array}{rl}
F(x,y,z,w)=g(x), & F(y,z,w,x)=g(y),\\
 F(z,w,x,y)=g(z), &  F(w,x,y,z)=g(w).
 \end{array}
\label{tfp}
\end{equation}
\label{defntfp}
\end{definition}

Notice that if $g$ is identity mapping, then Definition \ref{mmp} and
Definition\ref{defntfp}  reduce to Definition \ref{mmpa} and
Definition\ref{defntfpa}, respectively.
%%%%%%%%%%%%%%%%%%%%%%%%%%%%%%%%%%%%%%%%%%%%%%%%%%%%%%%%%%%%%%%%%%%%%%%%%
\begin{definition}
Let $F:X^4 \rightarrow X$ and $g: X \rightarrow X$.
$F$ and $g$ are called commutative if
\begin{equation}
g(F(x,y,z,w))= F(g(x),g(y),g(z),g(w)), \ \mbox{ for all } x,y,z,w \in X.\\
\label{commuting}
\end{equation}
\end{definition}

For a metric space $(X,d)$, the function $\rho:X^4 \times X^4 \rightarrow [0,\infty)$, given by,
\[\rho((x,y,z,w),(u,v,r,t)):=d(x,u)+d(y,v)+d(z,r)+d(w,t)\]
forms a metric space on $X^4$, that is, $(X^4,\rho)$ is a metric induced by $(X,d)$.

Let $\Phi$ denote the all functions $\phi:[0,\infty) \rightarrow [0,\infty)$ which is continuous and
satisfy that
\begin{itemize}
 \item[$(i)$] $\phi(t)<t$
\item[$(i)$] $\lim_{r \rightarrow t+} \phi(r)<t$ for each $ r>0$.
\end{itemize}

The aim of this paper is to prove the following theorem.
%%%%%%%%%%%%%%%%%%%%%%%%%%%%%%%%%%%%%%%%%%%%%%%%%%%%%%%%%%%%%%%%%
%%%%%%%%%%%%%%%%%%%%%%%%%%%%%%%%%%%%%%%%%%%%%%%%%%%%%%%%%%%%%%%%%%%%%%%%%
\begin{theorem}
Let $(X, \leq )$ be partially ordered set and $(X,d)$ be a complete metric space.
Suppose  $F:X^4 \rightarrow X$ and there exists $\phi\in \Phi$ such that
$F$ has the mixed $g$-monotone property and
\begin{equation}
d(F(x,y,z,w),F(u,v,r,t)) \leq  \phi\left(\frac{d(x,u)+d(y,v)+d(z,r)+d(w,t)}{4}\right)
\label{thm1eqmain}
\end{equation}
for all $x, u, y, v, z, r,w, t$ for which
$g(x)\leq g(u)$, $g(y)\geq g(v)$, $g(z)\leq g(r)$ and $g(w)\geq g(t)$.
Suppose there exist $x_0,y_0,z_0,w_0 \in X$ such that
\begin{equation}
\begin{array}{c}
g(x_0)\leq F(x_0,y_0,z_0,w_0),
\ \ g(y_0) \geq F(x_0,w_0,z_0,y_0), \\
\ \  g(z_0) \leq F(z_0,y_0,x_0,w_0),
\ \ g(w_0) \geq F(z_0,w_0,x_0,y_0).\\
\end{array}
\label{supposegg}
\end{equation}
Assume also that $F(X^4) \subset g(X)$ and $g$ commutes with $F$.
Suppose either
\begin{itemize}
\item[$(a)$] $F$ is continuous, or
\item[$(b)$] $X$ has the following property:
\begin{itemize}
\item[$(i)$] if non-decreasing sequence $x_n\rightarrow x$, then $x_n\leq x$ for all $n,$
\item[$(ii)$]if non-increasing sequence $y_n\rightarrow y$, then $y_n\geq y$ for all $n$,
\end{itemize}
\end{itemize}
then there exist $x,y,z,w \in X$ such that
\[
\begin{array}{c}
F(x,y,z,w)=g(x), \ \ \ \ F(x,w,z,y)=g(y), \\
F(z,y,x,w)=g(z), \ \ \ \ F(z,w,x,y)=g(w).
\end{array}
\]
\label{main}
that is, $F$ and $g$ have a common coincidence point.
\end{theorem}
%%%%%%%%%%%%%%%%%%%%%%%%%%%%%%%%%%%%%%%%%%%%%%%%%%%%%%%%%%%%%%%%%
%%%%%%%%%%%%%%%%%%%%%%%%%%%%%%%%%%%%%%%%%%%%%%%%%%%%%%%%%%%%%%%%%%%%%%%%%
\begin{proof}
Let $x_0,y_0,z_0,w_0 \in X$ be such that (\ref{supposegg}).
We construct the sequences $\{x_n\}$, $\{y_n\}$, $\{z_n\}$ and $\{w_n\}$ as follows
\begin{equation}
\begin{array}{c}
g(x_n)= F(x_{n-1},y_{n-1},z_{n-1},w_{n-1}),\\
g(y_n)=F(x_{n-1},w_{n-1},z_{n-1},y_{n-1}) ,\\
g(z_n)=F(z_{n-1},y_{n-1},x_{n-1},w_{n-1}),\\
g(w_n)=F(z_{n-1},w_{n-1},x_{n-1},y_{n-1}) .\\
\end{array}
\label{sequence}
\end{equation}
for $n = 1,2,3, ....$.\\
We claim that
\begin{equation}
\begin{array}{rl}
g( x_{n-1}) \leq g(x_n), &  g( y_{n-1}) \geq g(y_n),\\
g(z_{n-1}) \leq g(z_n), &  g( w_{n-1}) \geq g(w_n), \mbox{ for all } n\geq 1.\\
\end{array}
\label{sequence1}
\end{equation}
Indeed, we shall use mathematical induction to prove (\ref{sequence1}).
Due to (\ref{supposegg}), we have
\[
\begin{array}{c}
g(x_0)\leq F(x_0,y_0,z_0,w_0)=g(x_1),
\ \ g(y_0) \geq F(x_,w_0,z_0,y_0)=g(y_1), \\
\ \  g(z_0) \leq F(z_0,y_0,x_0,w_0)=g(z_1),
\ \ g(w_0) \geq F(z_0,w_0,x_0,y_0)=g(w_1).\\
\end{array}
\]
 Thus, the inequalities in (\ref{sequence1}) hold for $n=1$.
 Suppose now that the inequalities in (\ref{sequence1}) hold for some $n \geq 1$.
By mixed $g$-monotone property of $F$, together with (\ref{sequence}) and
(\ref{ommp}) we have
\begin{equation}
\begin{array}{c}
g(x_{n})=F(x_{n-1},y_{n-1},z_{n-1},w_{n-1}) \leq F(x_{n},y_{n},z_{n},w_{n}) =g(x_{n+1}), \\
g( y_{n})= F(x_{n-1},w_{n-1},z_{n-1},y_{n-1}) \geq F(x_{n},w_{n},z_{n},y_{n})=g(y_{n+1}),\\
g(z_{n})=F(z_{n-1},y_{n-1},x_{n-1},w_{n-1}) \leq F(z_{n},y_{n},x_{n},w_{n})=g(z_{n+1}),   \\
g( w_{n})=F(z_{n-1},w_{n-1},x_{n-1},y_{n-1}) \geq F(z_{n-1},w_{n-1},x_{n-1},y_{n-1})=g(w_{n+1}), \\
\end{array}
\label{sequence1a}
\end{equation}
Thus, (\ref{sequence1}) holds for all $n\geq 1$.
Hence, we have
\begin{equation}
\begin{array}{c}
\cdots g(x_n) \geq g(x_{n-1}) \geq \cdots \geq g(x_1) \geq g(x_0),\\
\cdots g(y_n)\leq g(y_{n-1})\leq \cdots \leq g(y_1) \leq g(y_0),\\
\cdots g(z_n) \geq g(z_{n-1}) \geq \cdots \geq g(z_1) \geq g(z_0),\\
\cdots g(w_n)\leq g(w_{n-1})\leq \cdots \leq g(w_1)\leq g(w_0),\\
\end{array}
\label{g_sequence}
\end{equation}

Set $\begin{array}{c}
\delta_n=d(g(x_n),g(x_{n+1}))+d(g(y_n),g(y_{n+1}))+d(g(z_n),g(z_{n+1}))\\
+d(g(w_n),g(w_{n+1}))
\end{array}$

We shall show that
\begin{equation}
\delta_{n+1} \leq 4\phi(\frac{\delta_{n}}{4}).
\label{g_delta}
\end{equation}

Due to (\ref{thm1eqmain}), (\ref{sequence}) and (\ref{g_sequence}), we have
\begin{equation}
\begin{array}{rl}
d(g(x_{n+1}),g(x_{n+2})) &= d(F(x_{n},y_{n},z_{n},w_n),F(x_{n+1},y_{n+1},z_{n+1},w_{n+1})) \\
& \phi\left(\frac{d(g(x_{n}),g(x_{n+1}))+d(g(y_{n}),g(y_{n+1}))+d(g(z_n),g(z_{n+1}))+d(g(w_n),g(w_{n+1}))}{4}\right)\\
& \leq \phi(\frac{\delta_{n}}{4})
\end{array}
\label{seqxn}
\end{equation}

\begin{equation}
\begin{array}{rl}
d(g(y_{n+1}),g(y_{n+2})) &= d(F(y_n,z_n,w_n,x_n),F(y_{n+1},z_{n+1},w_{n+1},x_{n+1})) \\
& \leq \phi\left(\frac{d(g(y_n),g(y_{n+1}))+d(g(z_n),g(z_{n+1}))+d(g(w_n),g(w_{n+1}))+d(g(x_n),g(x_{n+1}))}{4}\right)\\
& \leq \phi(\frac{\delta_{n}}{4})
\end{array}
\label{seqyn}
\end{equation}

\begin{equation}
\begin{array}{rl}
d(g(z_{n+1}),g(z_{n+2})) &= d(F(z_n,w_n,x_n,y_n),F(z_{n+1},w_{n+1},x_{n+1},y_{n+1})) \\
& \leq \phi\left(\frac{d(g(z_n),g(z_{n+1}))+d(g(w_n),g(w_{n+1}))+d(g(x_n),g(x_{n+1}))+ d(g(y_n),g(y_{n+1}))}{4}\right)\\
& \leq \phi(\frac{\delta_{n}}{4})
\end{array}
\label{seqzn}
\end{equation}

\begin{equation}
\begin{array}{rl}
d(g(w_{n+1}),g(w_{n+2})) &= d(F(w_n,x_n,y_n,z_n),F(w_{n+1},x_{n+1},y_{n+1},z_{n+1})) \\
& \phi\left(\frac{d(g(w_n),g(w_{n+1}))+d(g(x_n),g(x_{n+1}))+d(g(y_n),g(y_{n+1}))+d(g(z_n),g(z_{n+1}))}{4}\right)\\
& \leq \phi(\frac{\delta_{n}}{4})
\end{array}
\label{seqwn}
\end{equation}
Due to (\ref{seqxn})-(\ref{seqwn}), we conclude that
\begin{equation}
%\begin{array}{c}
d(x_{n+1},x_{n+2})+d(y_{n+1},y_{n+2})+d(z_{n+1},z_{n+2})+d(w_{n+1},w_{n+2})
  \leq 4 \phi(\frac{\delta_{n}}{4})
%\end{array}
\label{total1}
\end{equation}
Hence we have (\ref{g_delta}).

Since $\phi(t)<t$ for all $t>0$, then $\delta_{n+1}\leq \delta_{n}$ for all $n$.
Hence $\{\delta_n\}$ is a non-increasing sequence. Since it is bounded below, there is some $\delta\geq 0$ such that
\begin{equation}
\lim_{n\rightarrow \infty}\delta_n=\delta+.
\label{deltalimit1}
\end{equation}
We shall show that $\delta=0$. Suppose, to the contrary, that $\delta>0$.
Taking the limit as $\delta_n \rightarrow \delta+$ of both sides of
 (\ref{g_delta}) and having in mind that we suppose
$\lim_{t \rightarrow r} \phi(r)<t$ for all $t>0$,  we have
\begin{equation}
\delta=\lim_{n \rightarrow \infty}\delta_{n+1} \leq \lim_{n \rightarrow \infty} 4\phi(\frac{\delta_n}{4})
=\lim_{\delta_n \rightarrow \delta+} 4\phi(\frac{\delta_n}{4})<4\frac{\delta}{4}<\delta
\label{lim2}
\end{equation}
which is a contradiction.
Thus, $\delta=0$, that is,
\begin{equation}
\lim_{n\rightarrow \infty}[d(x_n,x_{n-1})+d(y_n,y_{n-1})+d(z_n,z_{n-1})+d(w_n,w_{n-1})] =0.
\label{deltalimit}
\end{equation}

Now, we shall prove that $\{g(x_n)\}$,$\{g(y_n)\}$,$\{g(z_n)\}$  and $\{g(w_n)\}$ are Cauchy sequences.
Suppose, to the contrary, that at least one of $\{g(x_n)\}$,$\{g(y_n)\}$,$\{g(z_n)\}$  and $\{g(w_n)\}$
is not Cauchy.  So, there exists an $\varepsilon>0$ for which we can find
subsequences $\{g(x_{n(k)})\}$, $\{g(x_{n(k)})\}$ of $\{g(x_{n})\}$ and
$\{g(y_{n(k)})\}$, $\{g(y_{n(k)})\}$ of $\{g(y_{n})\}$
and
$\{g(z_{n(k)})\}$, $\{g(z_{n(k)})\}$ of $\{g(z_{n})\}$
and
$\{g(w_{n(k)})\}$, $\{g(w_{n(k)})\}$ of $\{g(w_{n})\}$
with $n(k)>m(k)\geq k$ such that
\begin{equation}
\begin{array}{r}
d(g(x_{n(k)}),g(x_{m(k)}))
+d(g(y_{n(k)}),g(y_{m(k)}))\\
+d(g(z_{n(k)}),g(z_{m(k)}))
+d(g(w_{n(k)}),g(w_{m(k)}))\geq \varepsilon.
\end{array}
\label{lim3}
\end{equation}
Additionally, corresponding to $m(k)$, we may choose $n(k)$
such that it is the smallest integer satisfying (\ref{lim3})
and
$n(k)>m(k)\geq k$. Thus,
\begin{equation}
\begin{array}{r}
d(g(x_{n(k)-1}),g(x_{m(k)}))
+d(g(y_{n(k)-1}),g(y_{m(k)}))\\
+d(g(z_{n(k)-1}),g(z_{m(k)}))
+d(g(w_{n(k)-1}),g(w_{m(k)}))< \varepsilon.
\end{array}
\label{lim4}
\end{equation}
By using triangle inequality and having  (\ref{lim3}),(\ref{lim4}) in mind
\begin{equation}
\begin{array}{rl}
\varepsilon &\leq t_k=:
d(g(x_{n(k)}),g(x_{m(k)}))
+d(g(y_{n(k)}),g(y_{m(k)}))\\
&+d(g(z_{n(k)}),g(z_{m(k)}))
+d(g(w_{n(k)}),g(w_{m(k)})) \\
&
\leq d(g(x_{n(k)}),g(x_{n(k)-1})) + d(g(x_{n(k)-1}),g(x_{m(k)}))\\
&+d(g(y_{n(k))},g(y_{n(k)-1}))+d(g(y_{n(k)-1}),g(y_{m(k)}))\\
& \ \ + d(g(z_{n(k)}),g(z_{n(k)-1}))+d(g(z_{n(k)-1}),g(z_{m(k)}))\\
&+ d(g(w_{n(k)}),g(w_{n(k)-1}))+d(g(w_{n(k)-1}),g(w_{m(k)})) \\
&<
d(g(x_{n(k)}),g(x_{n(k)-1})) +
d(g(y_{n(k)}),g(y_{n(k)-1}))+\\
&d(g(z_{n(k)}),g(z_{n(k)-1}))+
d(g(w_{n(k)}),g(w_{n(k)-1}))+ \varepsilon.
\end{array}
\label{lim5}
\end{equation}
Letting $k \rightarrow \infty$ in (\ref{lim5}) and using (\ref{deltalimit})
\begin{equation}
\lim_{k \rightarrow \infty}t_k=\lim_{k \rightarrow \infty}
\left[\begin{array}{c}
d(g(x_{n(k)}),g(x_{m(k)}))+d(g(y_{n(k)}),g(y_{m(k)}))\\
+d(g(z_{n(k)}),g(z_{m(k)}))+d(g(w_{n(k)}),g(w_{m(k)}))
\end{array}
\right]=\varepsilon+
\label{lim6}
\end{equation}
Again by triangle inequality,
\begin{equation}
\begin{array}{rl}
t_k&=d(g(x_{n(k)}),g(x_{m(k)}))+d(g(y_{n(k)}),g(y_{m(k)}))\\
&+d(g(z_{n(k)}),g(z_{m(k)}))+d(g(w_{n(k)}),g(w_{m(k)})) \\
& \leq d(g(x_{n(k)}),g(x_{n(k)+1}))+d(g(x_{n(k)+1}),g(x_{m(k)+1}))+d(g(x_{m(k)+1}),g(x_{m(k)}))\\
&\ +d(g(y_{n(k)}),g(y_{n(k)+1}))+d(g(y_{n(k)+1}),g(y_{m(k)+1}))+d(g(y_{m(k)+1}),g(y_{m(k)}))\\
&\ +d(g(z_{n(k)}),g(z_{n(k)+1}))+d(g(z_{n(k)+1}),g(z_{m(k)+1}))+d(g(z_{m(k)+1}),g(z_{m(k)}))\\
&\ +d(g(w_{n(k)}),g(w_{n(k)+1}))+d(g(w_{n(k)+1}),g(w_{m(k)+1}))+d(g(w_{m(k)+1}),g(w_{m(k)}))\\
&\leq \delta_{n(k)+1}+\delta_{m(k)+1}+d(g(x_{n(k)+1}),g(x_{m(k)+1}))+d(g(y_{n(k)+1}),g(y_{m(k)+1}))\\
& \ \ +d(g(z_{n(k)+1}),g(z_{m(k)+1}))+d(g(w_{n(k)+1}),g(w_{m(k)+1}))
\end{array}
\label{lim7}
\end{equation}
Since $n(k)>m(k)$, then
\begin{equation}
\begin{array}{c}
g(x_{n(k)}) \geq g(x_{m(k)}) \mbox{ and } g(y_{n(k)}) \leq g(y_{m(k)}),\\
g(z_{n(k)}) \geq g(z_{m(k)}) \mbox{ and } g(w_{n(k)}) \leq g(w_{m(k)}).\\
\end{array}
\label{lim8}
\end{equation}
Hence from (\ref{lim8}), (\ref{sequence}) and (\ref{thm1eqmain}), we have,

\begin{equation}
\begin{array}{rl}
d(g(x_{n(k)+1}),g(x_{m(k)+1})) &=  d(F({x_{n(k)}},{y_{n(k)}},{z_{n(k)}},{w_{n(k)}}),
 F({x_{m(k)}},{y_{m(k)}},{z_{m(k)}},{w_{m(k)}}))\\
& \leq \phi \left( \begin{array}{l}
 \frac{1}{4}[d(g({x_{n(k)}}),g({x_{m(k)}})) + d(g({y_{n(k)}}),g({y_{m(k)}})) \\
  + d(g({z_{n(k)}}),g({z_{m(k)}})) + d(g({w_{n(k)}}),g({w_{m(k)}})) ]\\
\end{array}\right)
\end{array}
\label{seqxnk}
\end{equation}

\begin{equation}
\begin{array}{rl}
d(g(y_{n(k)+1}),g(y_{m(k)+1})) &=  d(F({y_{n(k)}},{z_{n(k)}},{w_{n(k)}},{x_{n(k)}}),
 F({y_{m(k)}},{z_{m(k)}},{w_{m(k)}},{x_{m(k)}})) \\
& \leq \phi \left( \begin{array}{l}
  \frac{1}{4}[d(g({y_{n(k)}}),g({y_{m(k)}})) + d(g({z_{n(k)}}),g({z_{m(k)}})) \\
  + d(g({w_{n(k)}}),g({w_{m(k)}})) + d(g({x_{n(k)}},{x_{m(k)}})) ]\\
 \end{array} \right)\\
\end{array}
\label{seqynk}
\end{equation}

\begin{equation}
\begin{array}{rl}
d(g(z_{n(k)+1}),g(z_{m(k)+1})) &= d(F({z_{n(k)}},{w_{n(k)}},{x_{n(k)}},{y_{n(k)}}),
 F({z_{m(k)}},{w_{m(k)}},{x_{m(k)}},{y_{m(k)}})) \\
& \leq \phi \left(  \frac{1}{4}[\begin{array}{l}
 d(g(z_{n(k)}),g({z_{m(k)}})) + d(g({w_{n(k)}}),g({w_{m(k)}})) \\
  + d(g({x_{n(k)}}),g({x_{m(k)}})) + dg(({y_{n(k)}}),g({y_{m(k)}})) ]\\
 \end{array} \right)\\
\end{array}
\label{seqznk}
\end{equation}

\begin{equation}
\begin{array}{rl}
d(g(w_{n(k)+1}),g(w_{m(k)+1}) &=
 d(F({w_{n(k)}},{x_{n(k)}},{y_{n(k)}},{z_{n(k)}}),
 F({w_{m(k)}},{x_{m(k)}},{y_{m(k)}},{z_{m(k)}}) \\
& \leq \phi \left( \begin{array}{l}
  \frac{1}{4}[d(g({w_{n(k)}}),g({w_{m(k)}})) + d(g({x_{n(k)}}),g({x_{m(k)}})) \\
  + d(g({y_{n(k)}}),g({y_{m(k)}})) + d(g({z_{n(k)}}),g({z_{m(k)}}))] \\
 \end{array} \right)\\
\end{array}
\label{seqwnk}
\end{equation}
Combining (\ref{lim7}) with (\ref{seqxnk})-(\ref{seqwnk}), we obtain that
\begin{equation}
\begin{array}{rl}
t_k &\leq \delta_{n(k)+1}+\delta_{m(k)+1}+d(g(x_{n(k)+1}),g(x_{m(k)+1})
+d(g(y_{n(k)+1}),g(y_{m(k)+1}))\\
& \ \ +d(g(z_{n(k)+1}),g(z_{m(k)+1}))+d(g(w_{n(k)+1}),g(w_{m(k)+1}))) \\
& \leq \delta_{n(k)+1}+\delta_{m(k)+1}+t_k+4\phi\left(\frac{t_k}{4}\right)\\
& < \delta_{n(k)+1}+\delta_{m(k)+1}+t_k+4\frac{t_k}{4}
\end{array}
\end{equation}
Letting $k \rightarrow \infty$, we get a contradiction.
This shows that  $\{g(x_n)\}$,$\{g(y_n)\}$ ,$\{g(z_n)\}$  and $\{g(w_n)\}$  are Cauchy sequences. Since $X$ is complete
metric space, there exists $x,y,z,w \in X$ such that
\begin{equation}
\begin{array}{c}
\lim_{n \rightarrow \infty}g(x_n) =x \mbox{ and } \lim_{n \rightarrow \infty}g(y_n) =y,\\
\lim_{n \rightarrow \infty}g(z_n )=z \mbox{ and } \lim_{n \rightarrow \infty}g(w_n) =w.\\
\end{array}
\label{elma}
\end{equation}

Since $g$ is continuous, (\ref{elma}) implies that
\begin{equation}
\begin{array}{c}
\lim_{n \rightarrow \infty}g(g(x_n)) =g(x) \mbox{ and } \lim_{n \rightarrow \infty}g(g(y_n)) =g(y),\\
\lim_{n \rightarrow \infty} g(g(z_n ))=g(z) \mbox{ and } \lim_{n \rightarrow \infty}g(g(w_n)) =g(w).\\
\end{array}
\label{elma_g}
\end{equation}

From (\ref{sequence1a}) and by regarding commutativity of $F$ and $g$,
\begin{equation}
\begin{array}{c}
g(g(x_{n+1})) =g(F(x_n,y_n,z_n,w_n))=F(g(x_n),g(y_n),g(z_n),g(w_n)), \\
g(g(y_{n+1})) =g(F(x_n,w_n,z_n,y_n))=F(g(x_n),g(w_n),g(z_n),g(y_n)), \\
g(g(z_{n+1})) =g(F(z_n,y_n,x_n,w_n))=F(g(z_n),g(y_n),g(x_n),g(w_n)), \\
g(g(w_{n+1})) =g(F(z_n,w_n,x_n,y_n))=F(g(z_n),g(w_n),g(x_n),g(y_n)), \\
\end{array}
\label{elma_gF}
\end{equation}

We shall show that
\[
\begin{array}{c}
F(x,y,z,w)=g(x), \ \ \ \ F(x,w,z,y)=g(y), \\
F(z,y,x,w)=g(z), \ \ \ \ F(z,w,x,y)=g(w).
\end{array}
\]
Suppose now  $(a)$ holds. Then by (\ref{sequence}),(\ref{elma_gF}) and (\ref{elma}),
we have
\begin{equation}
\begin{array}{rl}
g(x)&=\displaystyle\lim_{n \rightarrow \infty} g(g(x_{n+1})) = \lim_{n \rightarrow \infty} g(F(x_n,y_n,z_n,w_n))\\
&=\displaystyle \lim_{n \rightarrow \infty} F(g(x_n),g(y_n),g(z_n),g(w_n))\\
&=  F(\displaystyle\lim_{n \rightarrow \infty}g(x_n),
\displaystyle\lim_{n \rightarrow \infty} g(y_n),
\displaystyle \lim_{n \rightarrow \infty}g(z_n),
\displaystyle \lim_{n \rightarrow \infty}g(w_n))\\
&=F(x,y,z,w)\\
\end{array}
\label{sequencecont}
\end{equation}
Analogously, we also observe that
\begin{equation}
\begin{array}{rl}
g(y)&=\displaystyle\lim_{n \rightarrow \infty} g(g(y_{n+1})) = \lim_{n \rightarrow \infty} g(F(x_n,w_n,z_n,y_n)\\
&=\displaystyle \lim_{n \rightarrow \infty}F(g(x_n),g(w_n),g(z_n),g(y_n))\\
&=  F(\displaystyle\lim_{n \rightarrow \infty}g(x_n),
\displaystyle\lim_{n \rightarrow \infty} g(w_n),
\displaystyle \lim_{n \rightarrow \infty}g(z_n),
\displaystyle \lim_{n \rightarrow \infty}g(y_n))\\
&=F(x,w,z,y)\\
\end{array}
\label{sequencecont1}
\end{equation}
\begin{equation}
\begin{array}{rl}
g(z)&=\displaystyle\lim_{n \rightarrow \infty} g(g(z_{n+1})) = \lim_{n \rightarrow \infty} g(F(z_n,y_n,x_n,w_n))\\
&=\displaystyle \lim_{n \rightarrow \infty}F(g(z_n),g(y_n),g(x_n),g(w_n))\\
&=  F(\displaystyle\lim_{n \rightarrow \infty}g(z_n),
\displaystyle\lim_{n \rightarrow \infty} g(y_n),
\displaystyle \lim_{n \rightarrow \infty}g(x_n),
\displaystyle \lim_{n \rightarrow \infty}g(w_n))\\
&=F(z,y,x,w)\\
\end{array}
\label{sequencecont2}
\end{equation}
\begin{equation}
\begin{array}{rl}
g(w)&=\displaystyle\lim_{n \rightarrow \infty} g(g(w_{n+1})) = \lim_{n \rightarrow \infty} g(F(z_n,w_n,x_n,y_n))\\
&=\displaystyle \lim_{n \rightarrow \infty} F(g(z_n),g(w_n),g(x_n),g(y_n))\\
&=  F(\displaystyle\lim_{n \rightarrow \infty}g(z_n),
\displaystyle\lim_{n \rightarrow \infty} g(w_n),
\displaystyle \lim_{n \rightarrow \infty}g(x_n),
\displaystyle \lim_{n \rightarrow \infty}g(y_n))\\
&=F(z,w,x,y)\\
\end{array}
\label{sequencecont3}
\end{equation}

Thus, we have
\[
\begin{array}{c}
F(x,y,z,w)=g(x), \ \ \ \ F(y,z,w,x)=g(y), \\
F(z,,w,x,y)=g(z), \ \ \ \ F(w,x,y,z)=g(w).
\end{array}
\]
Suppose now the assumption $(b)$ holds. Since $\{g(x_n)\}, \ \{g(z_n)\}$ is non-decreasing and
$g(x_n) \rightarrow x, \ g(z_n) \rightarrow z$ and also
$\{g(y_n)\}, \ \{g(w_n)\}$ is non-increasing and
$g(y_n) \rightarrow y, \ g(w_n) \rightarrow$, then by assumption $(b)$ we have
\begin{equation}
g(x_n) \geq x, \ \ g(y_n) \leq y, \ \ g(z_n) \geq z, \ \ g(w_n) \leq w
\label{sondakka1}
\end{equation}
for all $n$.
Thus, by triangle inequality and (\ref{elma_gF})
\begin{equation}
\begin{array}{l}
d(g(x),F(x,y,z,w)) \leq d(g(x),g(g(x_{n+1})))+d(g(g(x_{n+1})),F(x,y,z,w))\\
\leq d(g(x),g(g(x_{n+1})))
+\phi\left(\frac{1}{4}\left[\begin{array}{c}d(g(g(x_n),g(x)))+d(g(g(y_n),g(y)))\\
+d(g(g(z_n),g(z)))+d(g(g(w_n),g(w)))\end{array}\right]\right)\\
\end{array}
\label{sondakka}
\end{equation}
Letting $n\rightarrow \infty$ implies that $d(g(x),F(x,y,z,w))\leq 0$.
Hence, $g(x)=F(x,y,z,w)$.
Analogously we can get that

\[F(y,z,w,x)=g(y), F(z,w,x,y)=g(z)  \mbox{  and  } F(w,x,y,z)=g(w).\]

Thus, we proved that $F$ and $g$ have a quartet coincidence point.
\end{proof}

%End of sections

%\section{check it}
\end{document}